\documentclass[12pt]{amsart}
\usepackage{pdfsync}
\usepackage[utf8]{inputenc}
\usepackage[T1]{fontenc}
\usepackage{amsthm,amsmath,amssymb,mathtools}
\usepackage{stmaryrd}
\usepackage[backend=bibtex,style=numeric,url=false,doi=false]{biblatex}
\bibliography{eigene,arxiv,mathscinet}
\usepackage{hyperref}
\usepackage{cleveref}
\usepackage[all]{xypic}
 \newtheorem{Thm}{Theorem}[section]
 \newtheorem{Lem}[Thm]{Lemma}
 \newtheorem{Prop}[Thm]{Proposition}
 \newtheorem{Cor}[Thm]{Corollary}
\theoremstyle{remark}

 \newtheorem{Rem}[Thm]{Remark}
\theoremstyle{definition}
 \newtheorem{Def}[Thm]{Definition}
\numberwithin{equation}{section}

\newcommand\Z{\mathbb Z}
\newcommand\CC{\mathbb C}
\newcommand\ol[1]{\overline{#1}}

\newcommand\adhit{\triangleright}
\newcommand\inv{^{-1}}
\def\HM#1.#2.#3.#4.{{^{#1}_{#3}\mathcal M^{#2}_{#4}}}

\newcommand\id{\operatorname{id}}

\newcommand\ot{\otimes}

\newcommand\Stab{\operatorname{Stab}}
\newcommand\Vect{\operatorname{Vect}}
\newcommand\Tr{\operatorname{Tr}}

\newcommand\C{\mathcal C}
\newcommand\CTR{\mathcal Z}
\newcommand\RCTR{\ol{\mathcal Z}}

\newcommand\ota{\odot}

\newcommand\otb{\diamond}
\newcommand{\oub}[1]{\underset{{#1}}{\otb}}

\newcommand{\co}[1]{\underset{{#1}}{\Box}}

\newcommand\nt{\diamond}
\newcommand\VecGom{\Vect_G^{\omega}}
\newcommand\VecGHom{\Vect_{G/H}^\omega}
\newcommand\VecG{\Vect_G}
\newcommand\VecH{\Vect_H}

\newcommand\GTHC{{_{\CC[H]}(\VecGom)_{\CC[H]}}}
\newcommand\AGTHC{{_{\CC[H]}(\VecGHom)}}

\newcommand\coc{\alpha}
\newcommand\ord{\operatorname{ord}}
\newcommand\Ker{\operatorname{Ker}}
\renewcommand\epsilon\varepsilon
\begin{document}
\title[Indicator formula for group-theoretical categories]{A Higher Frobenius-Schur Indicator Formula for Group-Theoretical Fusion Categories}
\author{Peter Schauenburg}
\address{Institut de Math{\'e}matiques de Bourgogne --- UMR 5584 du CNRS\\
Universit{\'e} de Bourgogne\\
Facult{\'e} des Sciences Mirande\\
9 avenue Alain Savary\\
BP 47870 21078 Dijon Cedex\\
France
}
\email{peter.schauenburg@u-bourgogne.fr}
\subjclass[2010]{18D10,16T05,20C15}
\keywords{Fusion category, Frobenius-Schur indicator}
\thanks{Research partially supported through a FABER Grant by the \emph{Conseil régional de Bourgogne}}
\begin{abstract}
  Group-theoretical fusion categories are defined by data concerning finite groups and their cohomology: A finite group $G$ endowed with a three-cocycle $\omega$, and a subgroup $H\subset G$ endowed with a two-cochain whose coboundary is the restriction of $\omega$.

The objects of the category are $G$-graded vector spaces with suitably twisted $H$-actions; the associativity of tensor products is controlled by $\omega$. Simple objects are parametrized in terms of projective representations of finite groups, namely of the stabilizers in $H$ of right $H$-cosets in $G$, with respect to  two-cocycles defined by the initial data.

We derive and study general formulas that express the higher Frobenius-Schur indicators of simple objects in a group-theoretical fusion category in terms of the group-theoretical and cohomological data defining the category and describing its simples.
\end{abstract}
\maketitle

\section{Introduction}
\label{sec:introduction}

A group-theoretical fusion category $\C(G,H,\omega,\psi)$ is defined in terms of a finite group $G$, a subgroup $H$, a $\CC^\times$-valued three-cocycle $\omega$ on $G$, and a two-cochain $\psi$ on $H$ whose coboundary is the restriction of $\omega$ to $H$. The category $\C(G,H,\omega,\psi)$ can be viewed as the category of bimodules in the category $\VecGom$ over the twisted group algebra $\CC_\psi[H]$, where $\VecGom$ is the category of $G$-graded vector spaces with associativity isomosphism given by the cocycle $\omega$. Group-theoretical fusion categories are a rather accessible class of fusion categories, introduced and studied by Ostrik \cite{MR1976233}, and given their name in the paper \cite{MR2183279}, which we also use as a general reference for the theory of fusion categories. As group-theoretical fusion categories are sometimes considered as the somewhat trivial case in the classification program for general fusion categories, it almost seems in order to recall that they form a large class with objects of rich and varied structure. They include Drinfeld doubles of finite groups, the twisted doubles of \cite{MR1128130}, and in fact so many examples of (module categories of) semisimple Hopf algebras that the question whether all semisimple complex Hopf algebras might be group-theoretical was open for some years before being answered negatively in \cite{MR2480712}.

Higher Frobenius-Schur indicators are invariants of an object in a pivotal fusion category (and hence also invariants of that category). They generalize, to higher degrees and more general objects, the degree two Frobenius-Schur indicator defined for a representation of a finite group by its namesakes in 1906. Categorical versions of degree two indicators were studied by Bantay \cite{MR1436801} and Fuchs-Ganchev-Szlachányi-Vecsernyés \cite{MR1657800}, indicators for modules over semisimple Hopf algebras were introduced by Linchenko-Montgomery \cite{MR1808131} and studied in depth by Kashina-Sommerhäuser-Zhu \cite{MR2213320}. The degree two indicators for modules over semisimple quasi-Hopf algebras were treated by Mason-Ng \cite{MR2104908}. Higher indicators for pivotal fusion categories were introduced in \cite{NgSch:CIHISQHA,NgSch:HFSIPC,NgSch:FSIESC}.

Higher Frobenius-Schur indicators are a useful invariant in the theory of fusion categories and (quasi)Hopf algebras, but they are not usually easy to calculate in specific examples. In \cite{Sch:CHFSIFCCIFG} we have established a general formula for the higher Frobenius-Schur indicators of simple objects in the fusion categories $\C(G,H,1,1)$, that is, for the case of fusion categories ``without cocycles''. We refer to the literature cited there for the many predecessors and models for such a formula studied for special cases. The main result of the present paper continues the work in \cite{Sch:CHFSIFCCIFG} by establishing a general formula valid for the simple objects in a general group-theoretical fusion category. The formula in \cite{Sch:CHFSIFCCIFG} contains a previously studied formula for the indicators in the Drinfeld center of a finite group as a special case, and similarly the general formula obtained in the present paper contains a (seemingly new) formula for twisted Drinfeld centers.

Simple objects of the fusion category $\C(G,H,1,1)$ are described by irreducible characters of the stabilizers of the right cosets of $H$ in $G$ under the action of $H$ on these cosets. The indicator formula in \cite{Sch:CHFSIFCCIFG} expresses the higher indicators in terms of these irreducible characters, and the combinatorics of the subgroups and cosets involved. Simple objects of the general fusion category $\C(G,H,\omega,\psi)$ are described by irreducible \emph{projective} characters of the same stabilizer subgroups, with respect to a cocycle determined by the cohomological data $\omega$ and $\psi$. The new indicator formulas duly replace characters by projective characters, but certain additional correction terms involving the cohomological data arise; they can roughly be attributed to the fact that in a projective representation the action of a power of a group element differs from the corresponding power of the action of the group element.

To obtain the indicator formulas we use the generalization obtained in \cite{NgSch:FSIESC} of the ``third formula'' of \cite{MR2213320}: The indicators of an object $M$ in a pivotal fusion category $\C$ can be computed as the trace of the powers of the ribbon structure on $K(M)$, where $K\colon \C\to \CTR(\C)$ is the adjoint to the underlying functor from the Drinfeld center of $\C$. This is particularly feasible for group-theoretical categories since the Drinfeld center of the bimodule category over an algebra in a monoidal category is often insensitive to the algebra and only ``sees'' the underlying category \cite{MR1822847}. This is true in particular in the case of $\C(G,H,\omega,\psi)$, which is a bimodule category in $\VecGom$; thus its Drinfeld center is the center of $\VecGom$, i.e.\ the module category of a Drinfeld double. This can be expressed as saying that a group-theoretical category $\C(G,H,\omega,\psi$ is Morita equivalent to $\VecGom$ (see the survey \cite{MR3077244} for the notion of Morita equivalence).

This general approach was already taken in \cite{Sch:CHFSIFCCIFG}; the main difficulty in extending the results to general group-theoretical categories lies, naturally, in dealing with cohomological data and conditions, and projective characters. Instead of trying to generalize the description of the adjoint functors in \cite{Sch:CHFSIFCCIFG} by restriction, induction, and twisting of the group characters that describe the simple objects both of the group-theoretical category and its center, we use a rather different technique to hide cohomological complications rather painlessly behind categorical machinery. All the categories we deal with can be expressed as categories of modules, over an algebra in a monoidal category $\mathcal A$, say, where the modules are taken to live in a different category, say $\mathcal B$, on which $\mathcal A$ acts. Under this translation, the forgetful functor from the Drinfeld center to the group-theoretical category splits in two rather natural-looking steps (restricting a module structure to a smaller algebra, and coarsening a grading), and its adjoint can be described similarly. A key trick we use along the way is a significant simplification due to Natale \cite{MR2196640} of the cohomological data defining a group-theoretical fusion category. She shows that any group-theoretical fusion category is equivalent to one of the form $\C(G,H,\omega,1)$ such that not only $\omega|_{H\times H\times H}=1$ (as follows from $\psi=1$), but already $\omega_{G\times G\times H}=1$ (and further conditions that we do not need). Thus every group-theoretical fusion category falls in the (therefore larger) class of categories studied in \cite{MR1887584}, and the quasi-Hopf algebras explicitly representing these categories constructed in \cite{MR1887584} are used in \cite{MR2196640} to find formulas for degree two Frobenius-Schur indicators. In the present paper, Natale's simplification is essential to bring the Frobenius-Schur indicator computations within the author's reach.

The paper's organization is as follows: In \cref{sec:preps} we recall categorical terminology and some results on the group-theoretical categories we study. Moreover, we show how to rewrite the relevant categories as module categories in a suitable sense, provided we are dealing with an ``adapted'' cocycle as provided by Natale. In \cref{sec:adjoints} we describe in detail the pair of adjoint functors between a group-theoretical category constructed from an adapted cocycle and its Drinfeld center. In \cref{sec:indicator-formula} we prove the indicator formula for group-theoretical categories defined by an adapted cocycle. \Cref{sec:indic-form-gener} discusses how to remove the condition that the cocycle be adapted, which is necessary to deal with twisted Drinfeld doubles in \cref{sec:indic-twist-doubl}, since the Drinfeld center of a pointed category $\VecGom$ is, in a natural way, a group-theoretical category whose cocycle is not adapted. 
\section{Preliminaries}
\label{sec:preps}

We denote the adjoint action of a group $G$ on itself by $x\adhit g=xgx\inv$. For a subgroup $H\subset G$ we call $gH$ a right coset.

Cohomology groups of a group $G$ are with coefficients in the trivial $G$-module $\CC^\times$. We apologize for admitting the notation $H^2(H,\CC^\times)$. Cochains are always normalized. The coboundary of $\alpha\in C^n(G,\CC^\times)$ is defined by $(d\alpha)(g,h,\dots)=\alpha(h,\dots)\alpha\inv(gh,\dots)\dots$. Ordinary and projective representations and their characters are over the field $\CC$; we refer to \cite{MR1215935} for projective representation theory.

\subsection{Categorical terminology}
\label{sec:categ-term}

The notions of $\C$-category (a.k.a.\ $\C$-module category, or $\C$-actegory), $\C$-bicategory, and structure preserving functor between such categories, are a well-established tool in the study of tensor categories. If we repeat some of these notions in this section, it is only to establish conventions (like the direction of associativity arrows).

The associativity constraint of a monoidal category $\C$ is $\phi\colon X\ot (Y\ot Z)\to (X\ot Y)\ot Z$. A left $\C$-category is a category $\mathcal M$ equipped with a functor $\nt\colon \C\times \mathcal D\to \mathcal D$ and a natural isomorphism $\psi\colon X\nt Y\nt V\to(X\ot Y)\nt V$ coherent with the associativity constraint of $\C$. We will make free use of the notion of module in $\mathcal D$ over an algebra $A$ in $\C$. If we need to emphasize the $\C$-category structure used in the definition of a module, we will sometimes write $_{A\nt}\mathcal D$ for the module category, otherwise just $_A\mathcal D$. If $\mathcal D,\mathcal E$ are $\C$-categories, a lax $\C$-functor is a functor $\mathcal F\colon\mathcal D\to\mathcal E$ equipped with a natural transformation $\xi\colon X\nt\mathcal F(V)\to \mathcal F(X\nt V)$ coherent with the associativity constraints for the respective $\C$-actions. It is a $\C$-functor if $\xi$ is an isomorphism. If a $\C$-functor $\mathcal F$ has a right adjoint $\mathcal G$, then $\mathcal G$ has a canonical structure $\xi'$ of lax $\C$-functor, determined by commutativity of the diagram
\begin{equation}
  \label{eq:1}
  \xymatrix{%
    \mathcal F(X\nt\mathcal G(V))\ar[r]^{\xi\inv}\ar[d]_{\mathcal F(\xi)}&X\nt\mathcal F\mathcal G(V)\ar[d]^{\epsilon}\\
      \mathcal F\mathcal G(X\nt V)\ar[r]^-{\epsilon}&X\nt V}.
\end{equation}

If $A$ is an algebra in $\C$, then a lax $\C$-functor $\mathcal F\colon \mathcal D\to\mathcal E$ induces a functor $_A\mathcal F\colon\mathcal D\to\mathcal E$. If $\mathcal F$ is a $\C$-functor and has a right adjoint $\mathcal G$, then the functor $_A\mathcal G$ is right adjoint to $_A\mathcal F$.

The left center $\CTR(\C)$ of a monoidal category $\C$ has objects $(V,c)$ where $V\in\C$, and $c\colon V\ot X\to X\ot V$ is a half-braiding. We will rather need the right center $\RCTR(\C)$ whose objects are pairs $(V,c)$ with a half-braiding $c\colon X\ot V\to V\ot X$.

\subsection{The twisted Drinfeld double of a group}
\label{sec:group-theor-categ}

The twisted double of a finite group $G$ equipped with a three-cocycle $\omega\colon G^3\to \CC^\times$ was constructed by Dijkgraaf, Pasquier and Roche \cite{MR1128130}. Majid \cite{MR1631648} explained the construction as an example of the center construction for tensor categories and introduced a sort of twisted Yetter-Drinfeld modules as a characterization of modules over the twisted double. Without ever using the quasi-Hopf algebra $D^\omega(G)$ that is the twisted double, we will use its module category, the Drinfeld center of graded vector spaces with twisted associativity. We will give some details to fix conventions and to replace Majid's Yetter-Drinfeld modules by modules in a suitable $\C$-category, an approach already taken in \cite{MR1897403} for the doubles of more general quasi-Hopf algebras. 

We denote by $\VecG$ the category of $G$-graded finite-dimensional vector spaces, and by $\VecGom$ that same category with the monoidal category structure given by the usual tensor product of graded vector spaces with the associativity constraint
\begin{equation*}
  \phi\colon U\ot(V\ot W)\ni u\ot v\ot w\mapsto \omega(|u|,|v|,|w|)u\ot v\ot w\in (U\ot V)\ot W;
\end{equation*}
here and in the sequel we assume tacitly that the elements $u,v,w$ are homogeneous, and denote by $|u|$ etc.\ their degrees in $G$.

Associated to $\omega$ we define the symbols
\begin{equation}
  \label{eq:2}
  \coc_g(x,y)=\omega(x,y,g)\omega\inv(x,y\adhit g,y)\omega(xy\adhit g,x,y)
\end{equation}
for $x,y,g\in G$. (Similar symbols are introduced in \cite{MR1128130}, and their properties are well-known; we merely vary sides and other conventions.) They satisfy the twisted cocycle condition
  \begin{equation}
    \label{eq:3}
      \coc_g(y,z)\coc_g(x,yz)=\coc_{z\adhit g}(x,y)\coc_g(xy,z).
  \end{equation}

\begin{Lem}
  $\VecG$ is a left $\VecG$-actegory with respect to the action $\otb$ defined by $X\otb V:=X\ot V$ as vector spaces, with grading defined by $|x\otb v|=|x|\adhit|v|$, and with associativity
  \begin{align}
    \label{eq:4}
    X\otb Y\otb V&\to (X\ot Y)\otb V\\
    x\otb y\otb v&\mapsto \coc_{|v|}(|x|,|y|)(x\ot y)\otb v.
  \end{align}
  A category equivalence $_{\C[G]\otb}(\VecG)\to\RCTR(\VecGom)$ is given by
  sending $V\in{_{\C[G]\otb}}(\VecG)$ to the underlying graded vector space of $V$ equipped with the half-braiding
  \begin{equation}
    \label{eq:5}
    c:X\ot V\ni x\ot v\mapsto |x|\cdot v\ot x,
  \end{equation}
     where the module structure $\mu\colon\C[G]\otb V\to V$ is given by $\mu(g\ot v)=g\cdot v$.

     The ribbon structure of $\RCTR(\Vect_G^\omega)$ is given by $\theta(v)=|v|.v$ for $v\in V\in{_{\CC[G]\otb}(\VecG)}$.
\end{Lem}
\begin{proof}
  Clearly $X\otb Y\otb V=(X\ot Y)\otb V$ as graded vector spaces. The coherence of the associator map amounts to the cocycle condition \eqref{eq:3}.

  For a vector space $V$, transformations $c\colon X\ot V\to V\ot X$ natural in $X\in\VecG$ are in bijection with maps $\mu\colon\CC[G]\ot V\to V$ by \eqref{eq:5}. If $V\in\VecG$, then $c$ is $G$-equivariant iff $\mu\colon\CC[G]\otb V\to V$ is equivariant. The hexagon equation for the half-braiding
  \begin{equation*}
    \xymatrix{X\ot(Y\ot V)\ar[rr]^-{\phi}\ar[d]_{X\ot c}&&(X\ot Y)\ot V\ar[d]^{c}\\
      X\ot(V\ot Y)\ar[d]_\phi&&V\ot(X\ot Y)\ar[d]^\phi\\
      (X\ot V)\ot Y\ar[rr]^{c\ot Y}&&(V\ot X)\ot Y}
  \end{equation*}
  is equivalent to the condition that $\mu$ be a $\CC[G]\otb$-module structure, since
  \begin{align*}
    \phi_{VXY}&c_{X\ot Y,V}\phi_{XYV}(x\ot y\ot v)\\
      &=\phi_{VXY}c_{X\ot Y,v}(\omega(|x|,|y|,|v|)x\ot y\ot v)\\
      &=\phi_{VXY}(\omega(|x|,|y|,|v|)|x||y|\cdot v\ot x\ot y)\\
      &=\omega(|x||y|\adhit v,|x|,|y|)\omega(|x|,|y|,|v|)|x||y|\cdot v\ot x\ot y\\
    (c_{XV}\ot Y)&\phi_{XVY}(X\ot c_{YV})(x\ot y\ot v)\\
      &=(c_{XV}\ot Y)\phi_{XVY}(x\ot |y|\cdot v\ot y)\\
      &=(c_{XV}\ot Y)(\omega(|x|,|y|\adhit|v|,|y|)x\ot |y|\cdot v\ot y)\\
      &=\omega(|x|,|y|\adhit|v|,|y|)|x|\cdot|y|\cdot v\ot x\ot y.
  \end{align*}
\end{proof}

\begin{Def}
  For $x\in G$ and $m\in\Z$ we define the invertible scalars $\pi_m(x)$ recursively by $\pi_0(x)=1$ and
  $\pi_{m+1}(x)=\omega(x,x^m,x)\pi_m(x)$. (Thus $\pi_{m-1}(x)=\pi_m(x)\omega\inv(x,x^{m-1},x)$.)
  
\end{Def}
We note that
\begin{equation}
  \label{eq:6}
  \omega(x,x^m,x)=\coc_x(x,x^m)=\coc_x(x^m,x).
\end{equation}

\begin{Lem}
  \begin{enumerate}
  \item Let $V\in{_{\CC[G]\otb}(\VecG)}$; denote by $\rho_V(g)$ the vector space automorphism of $V$ given by the action of $g$. Let $v\in V$ and $x=|v|$. Then $\rho_V(x)^m(v)=\pi_m(x)\rho_V(x^m)(v)$.
  \item Let $g,x\in G$. Then
    \begin{equation}
      \label{eq:7}
      \pi_m(x)=\pi_m(g\adhit x)\coc_x(g\adhit x^m,g)\coc_x\inv(g,x^m).
    \end{equation}
  \end{enumerate}
\end{Lem}
\begin{proof}
  The first assertion is well-known in the theory of projective representations (noting that the elements of degree $x$ form a projective representation of $\langle x\rangle$ with cocycle $\coc_x$): We have $\rho_V(x)^0(v)=\rho_V(x^0)(v)$, and $\rho_V(x)^{m+1}(v)=\rho_V(x)^m(x.v)=\pi_m(x)x^m.x.v=\pi_m(x)\coc_x(x^m,x)x^{m+1}.v$.

  For the second assertion we note first that for any $x\in G$ there is $V\in{_{\CC[G]\otb}(\VecG)}$ and $v\in V$ with $|v|=x$.  $\theta_V=\lambda\id_V$. Then
  \begin{align*}
    g.\theta^m_V(v)&=g.\rho_V(x)^m(v)\\
      &=\pi_m(x)g.x^m.v\\
    \theta^m_V(g.v)&=\rho_V(g\adhit x)^m(g.v)\\
      &=\pi_m(g\adhit x)(g\adhit x^m).g.v\\
      &=\pi_m(g\adhit x)\coc_x(g\adhit x^m,g)(g\adhit x^m)g.v\\
      &=\pi_m(g\adhit x)\coc_x(g\adhit x^m,g)gx^m.v\\
      &=\pi_m(g\adhit x)\coc_x(g\adhit x^m,g)\coc_x\inv(g,x^m)g.x^m.v
  \end{align*}
  and $\theta^m$ is a $\CC[G]$-module map.
\end{proof}

\subsection{Group-theoretical categories}
\label{sec:group-theor-categ}

Let $G$ be a finite group, $\omega\colon G^3\to\CC^\times$ a three-cocycle, $H\subset G$ a subgroup, and $\psi\colon H^2\to\CC^\times$ a two-cochain such that $d\psi=\omega$. Then the twisted group ring $\CC_{\psi}[H]$ is an associative algebra in $\VecGom$, and one can form the category $\C(G,H,\omega,\psi):={_{\CC_\psi[H]}(\VecGom)_{\CC_\psi[H]}}$ of bimodules in the category $\VecGom$. Endowed with the tensor product over $\CC_\psi[H]$ in the category, this becomes a fusion category; the categories thus constructed are called group-theoretical fusion categories \cite{MR2183279,MR1976233}. We will refer to the quadruple $(G,H,\omega,\psi)$ as \emph{group-theoretical data}.

In order to be able to apply the framework in \cite{MR1887584}, Natale has shown in \cite{MR2196640} that in the definition of a group-theoretical fusion category one can always choose a three-cocycle whose restriction to $H^3$ is trivial (rather than just a coboundary), and satisfying even more restrictive assumptions.
\begin{Def}
  Let $G$ be a finite group and $H\subset G$ a subgroup. We will call a three-cocycle $\omega$ on $G$ \emph{adapted} if $\omega|_{G\times G\times H}=1$.
\end{Def}
\begin{Prop}[Natale]
  Let $\omega$ be a three-cocycle on the finite group $G$ and $\psi$ a two-cochain on $H\subset G$ such that $\omega|_{H^3}=d\psi$. Then there is an adapted cocycle $\omega'$ cohomologous to $\omega$ such that
  \begin{equation}
    \label{eq:8}
    {_{\CC_\psi[H]}(\VecGom)_{\CC_\psi[H]}}\cong{_{\CC[H]}(\Vect_G^{\omega'})_{\CC[H]}}
  \end{equation}
  as fusion categories.
\end{Prop}

Backed up by Natale's result, we will now assume that the group-theoretical category we are dealing with is $_{\CC[H]}(\VecGom)_{\CC[H]}$ for an adapted cocycle $\omega$. Non-adapted cocycles will only return in \cref{sec:indic-form-gener}.

\begin{Rem}[cf.\cite{MR1887584,MR2196640}]
  Let $H\subset G$ be a subgroup, and $\omega$ an adapted three-cocycle on $G$.
  \begin{align}
    \label{eq:9}
    \forall x,y,z\in G,h\in H\colon&\omega(x,y,zh)=\omega(x,y,z)\\
    \forall g,x\in G,h\in\label{eq:10} H\colon&\coc_g(x,h)=\omega(x,h,g)=:\omega_g(x,h)\\
    \forall g,x\in G,h,y\in H\colon&\omega_{gh}(x,y)=\omega_g(x,y)\label{eq:11}
  \end{align}
\end{Rem}

\begin{Rem}
  $\VecGom$ being a monoidal category, $\VecG$ is a $\VecGom$-bicategory. Since $\VecH\subset\VecGom$ is a monoidal subcategory thanks to the triviality of $\omega|_{H^3}$, it follows that $\VecG$ is a $\VecH$-bicategory. While the right action of $\VecH$ is the canonical one with trivial associator, and the associator between left and right action is also trivial, all thanks to the assumption that $\omega$ is adapted, the left action, which we will denote by $\ota$, has the modified associativity
  \begin{align}
    \label{eq:12}
    X\ota Y\ota V&\to (X\ot Y)\ota V\\
    x\ota y\ota v&\mapsto\omega_{|v|}(|x|,|y|)(x\ot y)\ota v.
  \end{align}
  By construction we have  $_{\CC[H]\ota}(\VecG)_{\CC[H]}={_{\CC[H]}(\VecGom)_{\CC[H]}}$.
\end{Rem}

\begin{Lem}
  The $\VecH$-category structure on $\VecGom$ induces a $\VecH$-category structure on $\VecGHom$, which we define to be the category of finite dimensional vector spaces graded by $G/H$. We have an equivalence of $\VecH$-categories
  \begin{align*}
    Q\colon{(\VecGom)_{\CC[H]}}&\to{(\VecGHom)}\\
    M&\mapsto M/H\\
  \intertext{inducing a category equivalence}
    _{\CC[H]}Q\colon{_{\CC[H]}(\VecGom)_{\CC[H]}}&\to{_{\CC[H]}(\VecGHom)}
  \end{align*}
\end{Lem}

\section{Some pairs of adjoint functors}
\label{sec:adjoints}

We continue to work under the general assumption that we have group-theoretical data $(G,H,\omega,1)$ with a three-cocycle $\omega$ that is adapted for the subgroup $H\subset G$.

\begin{Rem}
  The functor $\mathcal F\colon\Vect_G\to\Vect_{G/H}$ defined by corestriction of the grading is a strict left $\VecH$-functor from $(\VecG,\otb)$ to $(\VecGHom,\ota)$ by \eqref{eq:10}.
\end{Rem}

\begin{Lem}
  There is a monoidal category equivalence \[T\colon{_{\CC[G]\otb}(\VecG)}\to\RCTR\left({_{\CC[H]}(\VecGom)_{\CC[H]}}\right)\] making the diagram
    \begin{equation*}
      \xymatrix{%
 _{\CC[G]\otb}(\VecG)\ar[r]^-T\ar[d]_{\mathcal R}&
 \RCTR\left({_{\CC[H]}(\VecGom)_{\CC[H]}}\right)\ar[dd]^{\mathcal U}\\
 _{\CC[H]\otb}(\VecG)\ar[d]_{_{\CC[H]}\mathcal F}\\
 _{\CC[H]\ota}(\VecGHom)  &\ar[l]^{_{\CC[H]}Q}{_{\CC[H]}(\VecGom)_{\CC[H]}}}
    \end{equation*}
    commute. Here $\mathcal U$ is the underlying functor, and $\mathcal R$ is defined by restricting module structures.
\end{Lem}
\begin{proof}
By \cite{MR1822847}, the center of the bimodule category $_{\CC[H]}(\VecGom)_{\CC[H]}$ is equivalent to the center of $\VecGom$.
The equivalence maps $V\in\RCTR(\VecGom)$ to $V\ot\CC[H]$ with the obvious right $\CC[H]$-module structure, the left $\CC[H]$-module structure
\begin{equation*}
  \CC[H]\ot V\ot\CC[H]
\xrightarrow{c\ot\CC[H]}V\ot\CC[H]\ot\CC[H]
\xrightarrow{V\ot\nabla}V\ot\CC[H],
\end{equation*}
and some half-braiding that we shall not need.
Here it would seem that, passing from \cite{MR1822847}, where results are formulated for strict monoidal categories, to the present concrete situation we have forgotten to put in parentheses in the triple tensor product, and associativity isomorphisms to switch them. However, associativities with the rightmost factor in $\VecH$ are trivial.

Since $_{\CC[G]\otb}(\VecG)\cong\RCTR(\VecGom)$, we have the desired equivalence, and its composition with the underlying functor maps $V$ to $V\ot\CC[H]$ with diagonal grading, obvious right $\CC[H]$-module structure endowed with the left $\CC[H]$-module structure
\begin{equation*}
  h\cdot(v\ot h')=(V\ot\nabla)(c(h\ot v)\ot h')=h\cdot v\ot hh'.
\end{equation*}
Thus, $Q\mathcal U T(V)=V$ with unchanged grading and $H$-action obtained from restricting the $G$-action.
\end{proof}

\begin{Prop}
  The right adjoint functor of the functor $\mathcal F\colon\VecG\to\VecGHom$ defined by coarsening the grading is given by cotensor product:
  \begin{align*}
    \mathcal G\colon\VecGHom&\rightarrow\VecG\\
    V&\mapsto \CC[G]\co{\CC[G/H]}V=\{\sum g_i\ot v_i\in \CC[G]\ot V|g_i\in|v_i|\}.
  \end{align*}
  $\mathcal G$ is a $\VecH$-functor with respect to
  \begin{align}
    \label{eq:13}
    \xi\colon X\otb(\CC[G]\co{\CC[G/H]}V)&\to\CC[G]\co{\CC[G/H]}(X\ota V)\\
    \notag x\ot g\ot v&\mapsto |x|\adhit g\ot x\ota v\
  \end{align}
  The functor $_{\CC[H]}\mathcal F\colon{_{\CC[H]}(\VecG)}\to{_{\CC[H]}(\VecGHom)}$ has the right adjoint $_{\CC[H]}\mathcal F$ that maps $V$ to $\CC[G]\co{\CC[G/H]}V$ endowed with the left $\CC[H]$-action
  $h(g\ot v)=h\adhit g\ot hv$.
\end{Prop}
\begin{proof}
  Gradings being comodule structures, the adjoint $\mathcal G$ is the well-known coinduction functor adjoint to corestricting comodule structures along a coalgebra map.
  
  The counit of adjunction is $\epsilon\colon\CC[G]\co{\CC[G/H]}V\ni g\ot v\mapsto v\in V$. Since $\mathcal F$ is a strict $\VecH$-functor, the adjoint has a unique lax $\VecH$-functor structure determined by commutativity of
  \begin{equation*}
    \xymatrix{%
      X\otb(\CC[G]\co{\CC[G/H]}V      )\ar@{=}[r]\ar[d]_{\xi}&X\ota(\CC[G]\co{\CC[G/H]}V)\ar[d]^{X\ota\epsilon}\\
      \CC[G]\co{\CC[G/H]}(X\ota V)\ar[r]^-{\epsilon}&X\ota V
    }
  \end{equation*}
Thus, $\xi(x\otb g\ot v)=|x|\adhit g\ot x\ota v$, and clearly $\xi$ is an isomorphism. If $V$ is a $\CC[H]$-module in $(\VecG,\ota)$, then $\mathcal G V$ is a $\CC[H]$-module in $\VecGHom$ by
\begin{equation}
  \label{eq:14}
  \CC[H]\otb(\CC[G]\co{\CC[G/H]}V)\xrightarrow{\xi}\CC[G]\co{\CC[G/H]}(\CC[H]\ota V)\xrightarrow{\CC[G]\ot\mu}\CC[G]\co{\CC[G/H]}V
\end{equation}
giving the claimed formula.
\end{proof}
\begin{Cor}We have a commutative diagram
    \begin{equation*}
      \xymatrix{%
 _{\CC[G]\otb}(\VecG)\ar[r]^-T&
 \RCTR\left({_{\CC[H]}(\VecGom)_{\CC[H]}}\right)\\
 _{\CC[H]\ota}(\VecGHom)\ar[u]^{L}  &\ar[l]^{_{\CC[H]}Q}{_{\CC[H]}(\VecGom)_{\CC[H]}}\ar[u]_{\ol K}}
    \end{equation*}
  where $\ol K$ denotes the two-sided adjoint to the underlying functor, and 
  \begin{align*}
    L\colon{_{\CC[H]}(\VecGHom)}&\to{_{\CC[G]\otb}(\VecG)}\\
    V&\mapsto \CC[G]\oub{\CC[H]}(\CC[G]\co{\CC[G/H]}V).
  \end{align*}
\end{Cor}

\section{Indicator formulas for adapted cocycles}
\label{sec:indicator-formula}

We continue to work with group-theoretical data $(G,H,\omega,1)$ in which $\omega$ is adapted.

In the sequel, for a vector space $V$ with some (twisted) action of a group $\Gamma$, we denote by $\rho_V(\gamma)$ the endomorphism of $V$ defined by $\rho_V(v)=\gamma\cdot v$. If $W$ is a vector space graded by the set $X$, and $Y\subset X$, we let $W_Y$ be the span of the homogeneous elements of $W$ with degrees in $Y$.

\begin{Prop}
  Let $M\in\GTHC$ correspond to $W=Q(M)\in\AGTHC$. Then the $m$-th Frobenius-Schur indicator of $M$ is
  \begin{equation}
    \label{eq:15}
    \nu_m(M)=\dfrac{1}{|H|}\sum_{\substack{x\in G\\x^m\in H}}\pi_{-m}(x)\Tr(\rho_{W_{xH}}(x^{-m}))
  \end{equation}
\end{Prop}
\begin{proof}
  Note first that the formula makes sense: If $x^m\in H$, then $x^m(xH)=xH$, and $\rho_{W_{xH}}(x^{-m})$ is well-defined.
  
  For $g\ot x\ot w\in\CC[G]\oub{\CC[H]}(\CC[G]\co{\CC[G/H]}W)$ with $g\in G$, $w\in W$ and $x\in|w|$ we have:
  \begin{equation}
    \label{eq:16}
    \theta^m_{L(W)}(g\ot x\ot v)=\pi_m(x)\coc_x(g,x^m)gx^m\ot x\ot v.
  \end{equation}
  In fact this is true for $m=0$, and if it is true for $m$ then
  \begin{align*}
    \theta^{m+1}(g\ot x\ot v)&=\theta^m((g\adhit x).(g\ot x\ot v))\\
    &=\theta^m(\coc_x(g\adhit x,g)(g\adhit x)g\ot x\ot v) \\
    &=\theta^m(\coc_x(g,x)gx\ot x\ot v)\\
    &=\pi_m(x)\coc_x(gx,x^m)\coc_x(g,x)gx^{m+1}\ot x\ot v\\
    &=\pi_m(x)\coc_x(g,x^{m+1})\coc_x(x,x^m)gx^{m+1}\ot x\ot v\\
    &=\pi_{m+1}\coc_x(g,x^{m+1})gx^{m+1}\ot x\ot v.
  \end{align*}
  If $R$ is a transversal of the right $H$-cosets in $G$, then
  \begin{equation}
    \label{eq:17}
    L(W)=\bigoplus_{\substack{g\in R\\x\in G}}\CC g\ot \CC x\ot W_{xH}.
  \end{equation}
  According to \eqref{eq:16}, the only terms of the sum \eqref{eq:17} mapped into themselves by $\theta^m$ are those where $x^m\in H$. But for $x^m\in H$ we have
  \begin{equation}
    \label{eq:18}
    \theta_{L(W)}^m(g\ot x\ot v)=\pi_m(x)g\ot x^m(x\ot v)=\pi_m(x)g\ot x\ot x^m.v
  \end{equation}
  and thus $\Tr(\theta^m_{\ol K(W)})=|G/H|\sum_{x,x^m\in H}\pi_m(x)\Tr(\rho_{W_{xH}}(x)^m)$.
\end{proof}

Let $h\in H$ and $w\in W_{xH}$. Then using \eqref{eq:7} we have, for $x\in G$ with $x^m\in H$:
\begin{align*}
  \pi_m(h\adhit x)&\rho_{W_{hxH}}((h\adhit x)^m)(h.w)\\
    &=\pi_m(h\adhit x)(h\adhit x)^m.h.w\\
    &=\pi_m(h\adhit x)\omega_x(h\adhit x^m,h)(h\adhit x^m)h.w\\
    &=\pi_m(h\adhit x)\omega_x(h\adhit x^m,h)hx^m.w\\
    &=\pi_m(h\adhit x)\omega_x(h\adhit x^m,h)\omega_x\inv(h,x^m)h.x^m.w\\
    &=\pi_m(h\adhit x)\coc_x(h\adhit x^m,h)\coc_x\inv(h,x^m)h.\rho_{W_{hxH}}(x^m)\\
    &=\pi_m(x)h.\rho_{W_{hxH}}(x^m),
\end{align*}
so that
\begin{equation}\label{eq:19}
  \pi_m(h\adhit x)\Tr(\rho_{W_{hxH}}(h\adhit x^m))=\pi_m(x)\Tr(\rho_{W_{xH}}(x^m)).
\end{equation}

\begin{Thm}\label{Thm:indic-form-adapt}
  Let $M\in\GTHC$ and $W=Q(M)\in\AGTHC$. Assume that the degrees of homogeneous elements of $M$ lie in the double coset $HgH$. Let $S=\Stab_H(gH)$, and let $\chi$ be the projective $\omega_g$-character of the projective $\omega_g$-Representation $W_{gH}$ of $S$. Let $\mathfrak R_g$ be a system of representatives for the orbits of the adjoint action of $S$ on $gH$. Then
  \begin{align}
    \label{eq:20}
    \nu_m(M)&=\frac 1{|S|}\sum_{\substack{r\in gH\\r^m\in S}}\pi_{-m}(r)\chi(r^{-m})\\
      \label{eq:21}&=\frac 1{|S|}\sum_{\substack{h\in H\\(gh)^m\in S}}\pi_{-m}(gh)\chi((gh)^{-m})\\
      \label{eq:22}&=\sum_{\substack{r\in\mathfrak R_g\\r^m\in S}}\frac 1{|S\cap C_G(r)|}\pi_{-m}(r)\chi(r^{-m}).
  \end{align}
\end{Thm}
\begin{proof}
  In the sum \eqref{eq:15}, the only nonzero terms have $x\in HgH$ now. Let $\mathfrak T$ be a transversal of the right cosets of $S$ in $H$. Then every element $x\in HgH$ has the form $t\adhit r$ for unique $r\in gH$ and $t\in\mathfrak T$. By \eqref{eq:19} we obtain \eqref{eq:20} and \eqref{eq:21}. The third version \eqref{eq:22} follows since $S\cap C_G(r)$ is the stabilizer of $r$ under the adjoint action of $S$ on $gH$.
\end{proof}

\section{Indicator formulas for general cocycles}
\label{sec:indic-form-gener}

So far, our indicator formulas were derived, and worked, only for adapted cocycles. We will now write out the form that the formulas take for an arbitrary three-cocycle on $G$ which is not necessarily trvial, but only cohomologically so, on $H$.

We start by repeating a version of Ostrik's description \cite{MR1976233} of the simple objects in a group-theoretical fusion category in the general case. Thus, we are now dealing with a group $G$, subgroup $H$, and three-cocycle $\omega\colon G^3\to\CC^\times$, and with a normalized two-cochain $\psi\colon H^2\to\CC^\times$ such that $\omega|_{H^3}=d\psi$.

Let $M$ be an object of $_{\CC_\psi[H]}(\VecGom)_{\CC_\psi[H]}$ such that the degrees of its nonzero homogeneous elements are in the double coset $HgH$. Then $M$ can be equivalently described by a projective representation of $S=\Stab_H(gH)$ as follows: Let $M_g$ be the homogeneous component of degree $g$, and define a left action of $S$ on $M_g$ by
\begin{equation}
  \label{eq:23}
  s*m=(s.m).(g\inv\adhit s)
\end{equation}
Since
\begin{align*}
  s*t*m&=(s.((t.m).g\inv\adhit t)).g\inv\adhit s\\
    &=\omega(s,tg,g\inv\adhit t\inv)(s.t.m).g\inv\adhit t\inv.g\inv\adhit s\inv\\
    &=
    \begin{multlined}[t]
\omega(s,tg,g\inv\adhit t\inv)\omega(s,t,g)\psi(x,y)\\\omega\inv(stg,g\inv\adhit y\inv,g\inv\adhit x\inv)\psi(g\inv\adhit t\inv,g\inv\adhit s\inv)\\(st.m).d\inv\adhit(xy)\inv,
\end{multlined}
\end{align*}
we see that $M_g$ is a projective representation of $S$ with respect to the two-cocycle on $S$ given by
\begin{align}
  \label{eq:24}
  \beta_g(s,t)&
  \begin{multlined}[t]
=\psi(s,t)\psi(g\inv\adhit t\inv,g\inv\adhit s\inv)\omega(s,tg,g\inv\adhit t\inv)\\\omega(s,t,g)\omega\inv(stg,g\inv\adhit t\inv,g\inv\adhit s\inv).
\end{multlined}
\end{align}

If $\omega=d\eta$ and $\psi=\eta|_{H^2}$ for a two-cochain $\eta$ on $G$, then
\begin{align*}
  \beta_g(s,t)&=\begin{multlined}[t]\eta(s,t)\eta(g\inv\adhit t\inv,g\inv\adhit s\inv)\eta(tg,g\inv\adhit t\inv)\\\eta\inv(stg,g\inv\adhit t\inv)\eta(s,tg(g\inv\adhit t\inv))\eta\inv(s,tg)\eta(t,g)\\\eta\inv(st,g)\eta(s,tg)\eta\inv(s,t)\eta\inv(g\inv\adhit t\inv,g\inv\adhit s\inv)\\\eta(stg(g\inv\adhit t\inv),g\inv\adhit s\inv)\eta\inv(stg,(g\inv\adhit t\inv)(g\inv\adhit s\inv))\\\eta(stg,g\inv\adhit t\inv)
  \end{multlined}
  \\
  &=\begin{multlined}[t]\eta(tg,g\inv\adhit t\inv)\eta(s,g)\eta(t,g)\eta\inv(st,g)\\\eta(sg,g\inv\adhit s\inv)\eta\inv(stg,g\inv\adhit(st)\inv)
  \end{multlined}
  \\
  &=(d\lambda)(s,t)
\end{align*}
where $\lambda(s)=\eta(sg,g\inv\adhit s\inv)\eta(s,g)$.

Moreover, for $\omega=d\eta$, the scalars $\pi_m(x)$ simplify to $\pi_m(x)=\eta(x,x^m)\eta\inv(x^m,x)$: Clearly this is true for $m=0$; if it is true for $m$, then
\begin{align*}
  \pi_{m+1}(x)&=\pi_m(x)(d\eta)(x,x^m,x)\\
    &=\eta(x,x^m)\eta\inv(x^m,x)\eta(x^m,x)\eta\inv(x^{m+1},x)\eta(x,x^{m+1})\eta\inv(x,x^m)\\
    &=\eta(x,x^{m+1})\eta\inv(x^{m+1},x)\\
  \pi_{m-1}(x)&=\pi_m(x)(d\eta\inv)(x,x^{m-1},x)\\
    &=\eta(x,x^m)\eta\inv(x^m,x)\eta\inv(x^{m-1},x)\eta(x^m,x)\eta\inv(x,x^m)\eta(x,x^{m-1})\\
    &=\eta(x,x^{m-1})\eta\inv(x^{m-1},x).
\end{align*}

If $\omega=1$ and $\psi=d\theta$ for $\theta\colon H\to \CC^\times$ then we find $\beta_g(s,t)=d\gamma$ for $\gamma(s)=\theta(s)\theta(g\inv\adhit s\inv)$.

If $\omega'=\omega(d\eta)$ and $\psi'=\psi(\eta|_{H^2})(d\theta)$ for a two-cochain $\eta$ on $G$ and a one-cochain $\theta$ on $H$, then the group-theoretical categories associated to $(G,H,\omega,\psi)$ and $(G,H,\omega',\psi')$, respectively, are equivalent; the associated two-cocycles $\beta_g$ and $\beta'_g$ are related by $\beta_g'=\beta_g(d\lambda)(d\gamma)$, and the scalars $\pi_m(x)$ and $\pi_m'(x)$ by $\pi'_m(x)=\eta(x,x^m)\eta\inv(x^m,x)\pi_m(x)$. If $\chi$ is a projective character of $S=\Stab_H(gH)$ with cocycle $\beta_g$, then the associated projective character corresponding to $\beta_g'$ is $\chi'=\chi\lambda\gamma$.

\begin{Thm}\label{Thm:gencoc}
  Let $M$ be the simple object of the group-theoretical category $\C(G,H,\omega,\psi)$ associated to the $\beta_g$-projective character $\chi$ of the stabilizer $S=\Stab_H(gH)$. Let $\eta$ be a two-cochain on $G$ such that $\omega(d\eta)$ is adapted, and let $\theta$ be a one-cochain on $H$ such that $\psi(\eta|{H^2})(d\theta)=1$. Then the higher Frobenius-Schur indicators of $M$ are given by
  \begin{align}
    \label{eq:25}
    \nu_m(M)&=\frac 1{|S|}\sum_{\substack{r\in gH\\r^m\in S}}\tilde\pi_{-m}(r)\chi(r^{-m})\\
      \label{eq:26}&=\frac 1{|S|}\sum_{\substack{h\in H\\(gh)^m\in S}}\tilde\pi_{-m}(gh)\chi((gh)^{-m})\\
      \label{eq:27}&=\sum_{\substack{r\in\mathfrak R_g\\r^m\in S}}\frac 1{|S\cap C_G(r)|}\tilde\pi_{-m}(r)\chi(r^{-m})\\
      \intertext{where}
\tilde\pi_{m}(s)&=\begin{multlined}[t]\pi_m(s)\eta(s,s^{m})\eta\inv(s^{m},s)\\\eta(s^{m}g,g\inv\adhit s^{-m})\eta(s^{m},g)\theta(s^{m})\theta(g\inv\adhit s^{-m}).
\end{multlined}     
  \end{align}
\end{Thm}
\begin{proof}
  This results from applying \cref{Thm:indic-form-adapt} to the adapted cocycle $\omega(d\eta)$, in view of the above calculations.
\end{proof}

Recall (see \cite{MR1002038}) that the third cohomology group of the cyclic group $\mathbb Z/N\mathbb Z$ of order $N$ is generated by the cocycle $\underline\kappa$ defined by
\begin{equation}
\label{eq:28}  \underline\kappa(\ol j,\ol k,\ol\ell)=\exp\left(\displaystyle\frac{2\pi i}{N^2}[\ell]\left([j]+[k]-[j+k]\right)\right),
\end{equation}
where $[k]\equiv k(N)$ and $[k]\in\{0,\dots,N-1\}$. We note for later use that for $\omega=\underline\kappa$ we obtain $\alpha_{\ol\ell}(\ol j,\ol k)=\underline\kappa(\ol j,\ol k,\ol\ell)=(d\underline\lambda_{\ol k})(\ol j,\ol k)$, where $\lambda_{\ol k}(\ol j)=\exp\left(\dfrac{2\pi i}{N^2}[k][j]\right)$.

\begin{Prop}
  Let $(G,H,\omega,\psi)$ be group-theoretical data. Let $\omega'=\omega\kappa^t$, where $p\colon G\to \mathbb Z/N\mathbb Z$ is a group homomorphism with $H\subset\Ker(p)$, and $\kappa$ is inflated from $\underline\kappa$ along $p$. 

  Then simple objects of $\C(G,H,\omega',\psi)$ (associated to the double coset $HgH$) are in one-to-one correspondence with simple objects of $\C(G,H,\omega,\psi)$ (associated to the same double coset).

  Let $\zeta=\exp\left(\displaystyle\frac{2\pi i}{N}\right)$.

  If the simple object $M'$ of $\C(G,H,\omega',\psi)$ corresponds to $M\in\C(G,H,\omega,\psi)$ which is associated to the double coset $HgH$ with $p(g)=\ol k$, then
  \begin{equation}
    \label{eq:29}
    \nu_{m}(M')=\underline\kappa^{ts}(p(g),p(g\inv),p(g))\nu_{m}(M)=\zeta^{-stk^2/\gcd(k,N)}\nu_{m}(M)
  \end{equation}
  if $m=s\ord(p(g))=sN/\gcd(k,N)$,
  and $\nu_m(M)=\nu_m(M')=0$ if $\ord(p(g))=N/\gcd(k,N)$ does not divide $m$.
\end{Prop}
\begin{proof}
  It suffices to treat the case $t=1$. More than adapted, $\kappa$ vanishes whenever one of its arguments is in $H$ (or the kernel of $p$, for that matter). This implies that the analog of $\beta_g$ associated to $\kappa$ would be trivial, and thus that $\beta'_g=\beta_g$.

  Let $M$ and $M'$ as above correspond to the $\beta_g$-character $\chi$ of $S=\Stab_H(gH)$. If $g\in\Ker(p)$, then obviously nothing changes between the indicator formulas for $M$ and $M'$. Let $g\not\in\Ker(p)$ and $e=\ord(p(g))$. If $e\nmid m$, then for all $x\in H$ we have $p((gx)^m)=p(g)^mp(x)^m\not=p(x^m)$, and thus in particular $(gx)^m\not=x^m$. Thus $\nu_m(M)=\nu_m(M')=0$. Now let $m=es$. Let $\underline\pi_m$ be the analog of $\pi_m$ defined for the three-cocycle $\kappa$. Since $p(gx)=p(g)$ and $p((gx)^k)=p(g)^k$, we have $\underline\pi_m(gx)=\underline\pi_m(g)$ independent of $x$. Now let $p(g)=\ol k$, so $e=N/\gcd(k,N)$. For any $a\in\Z$ we have
  \begin{align*}
    \prod_{r=1}^e\underline\kappa(\overline k,(a+r)\overline k,\overline k)&=\exp\frac{2\pi i}{N^2}\left(k\sum_{r=1}^e(k+[(a+r)k]-[(a+r-1)k])\right)\\
    &=\exp\frac{2\pi i}{N^2}k\left(ek+[(a+e)k]-[ak]\right)\\
    &=\exp\frac{2\pi i}{N^2}ek^2\\
    &=\exp\frac{2\pi i}N\frac{k^2}{\gcd(k,N)}.
  \end{align*}
  Thus, if we denote by $\tilde\pi_m'$ the analog of $\tilde\pi_m$ defined for $\omega'$ instead of $\omega$, then when $m=se$, then
  $\tilde\pi'_m(g)=\tilde\pi_m(g)\zeta^{sk^2/\gcd(k,N)}$. Finally note that $\underline\kappa(p(g),p(g)\inv,p(g))=\zeta^k$.
\end{proof}

Thanks to the correction terms for ``bad'' cocycles amassed in the definition of $\tilde\pi_m$, the formulas in \cref{Thm:gencoc} may be tedious to apply. At first sight, this seems to be particularly true since passing from a cocycle $\omega$ whose restriction to $H$ is trivial to an adapted cocycle as indicated by Natale is a two-step process with quite an involved definition of a two cochain $\eta$. If we revisit Natale's formulas, however, it turns out that we get fairly lucky:

Assume that the three-cocycle $\omega$ satisfies $\omega|_{H^3}=1$. Choose a cross section $Q$ of the right $H$-cosets in $G$. Then, following \cite[Prop.4.2]{MR2196640} with switched sides, $\eta_1(ph,qh')=\omega(p,h,h')$ defines a cochain on $G$ such that $\eta_1|_{H\times G}=1$, and $\omega_0|_{G\times H\times H}=1$ for $\omega_0=\omega(d\eta_1)$. In the next step, define $\eta_2(ph,qh')=\omega_0(ph,q,h')\omega_0\inv(p,h,q)$. Then $\eta_2|_{G\times H}=1$ and $\eta_2|_{H\times Q}=1$, and $\omega_0(d\eta_2)=\omega(d\eta)$ for $\eta=\eta_1\eta_2$ is adapted. (The next step taken by Natale is not necessary for our purposes.)
Computing indicator values, we can of course assume that our double coset representatives are in $Q$. Thus if $s^m\in\Stab_H(dH)$ then 
\begin{align*}\tilde\pi_m(s)&=\pi_m(s)\eta_1(s^md,d\inv\adhit s^{-m})\\
  &=\pi_m(s)\eta_1(d(d\inv\adhit s^m),d\inv\adhit s^{-m})\\
  &=\pi_m(s)\omega(d,d\inv\adhit s^m,d\inv\adhit s^{-m}).
\end{align*}

Thus
\begin{Cor}\label{cor:indic-triv-res}
  Let $\omega$ be a three-cocycle on $G$ whose restriction to $H$ is trivial. Let $M$ be the simple object of the group-theoretical category $\C(G,H,\omega,1)$ associated to the $\beta_g$-projective character $\chi$ of the stabilizer $S=\Stab_H(gH)$. Then
  \begin{equation}
    \label{eq:30}
    \nu_m(M)=\frac1{|S|}\sum_{\substack{r\in gH\\r^m\in S}}\pi_{-m}(r)\omega(g,g\inv\adhit r^{-m},g\inv\adhit r^m)\chi(r^{-m}).
  \end{equation}
\end{Cor}
\section{Indicators of twisted doubles}
\label{sec:indic-twist-doubl}

\newcommand\G{G}
\newcommand\w{\omega}
\newcommand\GG{\Gamma}
\newcommand\ww{\varpi}
\newcommand\HH{H}

The module category of the twisted double of a finite group is a particular example of a group-theoretical fusion category. Let $\G$ be a finite group, and $\w$ a three-cocycle on $\G$. Let $\GG=\G\times\G$ and $\ww$ the three-cocycle on $\GG$ defined by $\ww((x,f),(y,g),(z,h))=\w(x,y,z)\w\inv(f,g,h)$. Put $\HH=\{(x,x)|x\in \G\}$; then $\ww|_{\HH^3}=1$. Ostrik observed \cite{MR1976233} that the module category of the twisted double $D^\w(\G)$ is equivalent to $\C(\GG,\HH,\ww,1)$. We'll identify $\HH$ with $\G$ in the sequel.

We will give an expression for the indicators of $D^\w(\G)$-modules using the formulas for indicators of objects in group-theoretical categories, in particular $\C(\GG,\HH,\ww,1)$. It remains to bridge the last gap between the two categories, however, by making the equivalence very explicit.

First, note that an object in $_{\CC[\G]\otb}(\Vect_{\G})$ decomposes as the direct sum of subobjects the degrees of whose homogeneous elements lie in a conjugacy class of $\G$. An object the degree of whose homogeneous elements lies in the conjugacy class of $g$ can equivalently be described by the action of the centralizer $C_{\G}(g)$ on its $g$-homogeneous component. This in turn is an $\coc_g$-projective representation of $C_{\G}(g)$.

A set of representatives of the double cosets of $\HH$ in $\GG$ is given by $(g,1)$ where $g$ runs through a set of representatives of the conjugacy classes of $\G$. The stabilizer of $(g,1)\HH$ is the centralizer $C_\G(g)$. Thus simple objects in $\C(\GG,\HH,\ww,1)$ are described by $\beta_{(g,1)}$-projective representations of $C_\G(g)$, where the two-cocycle $\beta_{(g,1)}$ is defined in terms of the three-cocycle $\ww$.

To explicitly pass between the two types of projective representations we use the structure theorem for Hopf modules over quasi-Hopf algebras \cite{Sch:HMDQHA}, which asserts that we have a monoidal category equivalence
\begin{equation}
  \label{eq:31}
  _{\CC[G]\otb}(\Vect_{\CC[\G]})\ni V\mapsto V\ot\CC[\G]\in{_{\CC[G]}(\Vect_{\G\times\G}^\ww)_{\CC[G]}},
\end{equation}
where $V\ot\CC[\G]$ has the $\GG$-grading $|v\ot g|=(|v|g,g)$, the right action $(v\ot g)x=\w\inv(|v|,g,x)v\ot gx$ and the left action $x(v\ot g)=\w(x,|v|,g)\w(x\adhit|v|,x,g)x.v\ot xg$. With this, the left action of $C_\G(g)$ on $(V\ot\CC[G])_{(g,1)}$ is given by
\[x*(v\ot 1)=(x.(v\ot 1)).x\inv=\w\inv(g,x,x\inv)x.v\ot 1.\]

In particular, the $\beta_{(g,1)}$-character corresponding to the $\coc_g$-character $\chi$ is given by $\chi\lambda_g$ with $\lambda_g(x)=\w\inv(g,x,x\inv)$, and $\beta_{(g,1)}$ and $\coc_g$ differ by the boundary of $\lambda_g$.

\begin{Thm}
  Let $\G$ be a finite group, and $\w$ a three-cocycle on $\G$. Let the simple object $M\in{_{\CC[G]\otb}(\Vect_{\G})}$ correspond to the $\coc_g$-character $\chi$ of $C_\G(g)$. Then
    \begin{equation}
      \label{eq:32}
      \nu_m(M)=\frac 1{|C_\G(g)|}\sum_{\substack{x\in G\\(gx)^m=x^m}}\frac{\pi_{m}(gx)}{\pi_{m}(x)}\chi(x^{m}).
    \end{equation}
\end{Thm}
\begin{proof}
  We apply \cref{cor:indic-triv-res} to the group $\GG$ and the three-cocycle $\ww$, whose restriction to $\G$ is in fact trivial.
  
  We note that in general $\nu_{-m}(V)=\ol{\nu_m(V)}$ and in a braided category indicators are real, which allows us to forget the minus signs on the degree $m$ of the indicators. The analog of $\pi_m((g,1)(x,x))$ for $\ww$ is $\pi_m(gx)/\pi_m(x)$, and the $\w$-factor in \cref{cor:indic-triv-res} disappears thanks to the passage from $\alpha_g$-projective characters and $\beta_{(g,1)}$-projective characters discussed before the statement of the theorem.
\end{proof}
\begin{Prop}
  Let $\G$ be a finite group, and $\w\colon \G^3\to \CC^\times$ a three-cocycle. Let $N\leq G$ be a subgroup of index two, and let $\w'$ be the product of $\w$ and the inflation of the nontrivial three-cocycle on $\G/N$. Then there is a one-to-one correspondence between objects of $\RCTR(\VecGom)$ and $\RCTR(\Vect_{\G}^{\w'})$. If $M\in\RCTR(\VecGom)$ corresponds to $M'\in\RCTR(\Vect_{\G}^{\w'})$, both are associated to the same conjugacy class $g^\G$. We have $\nu_m(M')=\nu_m(M)$ unless $g\not\in N$ and $m\equiv 2(4)$, in which case $\nu_m(M')=-\nu_m(M)$.
\end{Prop}
\begin{proof}Denote the inflated cocycle by $\kappa$.
  If $g\in N$ and $x\in G$ then $\kappa(gx,(gx)^m,gx)=\kappa(x,x^m,x)$, so that $\pi_m(gx)/\pi_m(x)=\pi_m'(gx)/\pi'_m(x)$, independent of $m$. Also, $\alpha'_g(x,y)=\alpha_g(x,y)$ for $x,y\in C_G(g)$, so that $\alpha_g'$-representations and $\alpha_g$-representations of $C_G(g)$ are the same.

  Now let $g\not\in N$. Then if $m$ is even, $\kappa(gx,(gx)^m,gx)=1=\kappa(x,x^m,x)$. If $m$ is odd, then $\kappa(gx,(gx)^m,gx)=-\kappa(x,x^m,x)=\pm 1$. Thus
  \begin{equation*}
    \pi_m'(gx)/\pi'_m(x)=
    \begin{cases}
      \pi_m(gx)/\pi_m(x)&\text{ if }m\equiv 0(4)\text{ or }m\equiv 1(4), \\
      -\pi_m(gx)/\pi_m(x)&\text{ if }m\equiv 2(4)\text{ or }m\equiv 3(4).
    \end{cases}
  \end{equation*}
  Furthermore, $\alpha_g'$ and $\alpha_g$ agree on $C_G(g)\times C_G(g)$ up to the coboundary of $\lambda_g$ defined by $\lambda_g(x)=i$ if $x\not\in N$ and $\lambda_g(x)=1$ if $x\in N$. Thus $\alpha_g$-characters $\chi$ and $\alpha_g'$-characters $\chi'$ of $C_G(g)$ are in one-to-one correspondence via $\chi'=\lambda_g\chi$. Since $g\not\in N$, the sum \eqref{eq:32} is empty unless $m$ is even, in which case $\lambda_g(x^{-m})=1$. 
\end{proof}

\begin{Rem}
 \cite{MR2333187} contains a list of the higher Frobenius-Schur indicators for the twisted doubles of groups of order eight. One can see in these tables that replacing the three-cocycle by its product with a cocycle inflated from a factor group of order two does indeed change the sign of the indicators of degree congruent to two modulo four, and this for eight of the simple objects; indeed of the five congruence classes in the groups considered, two are not in the respective subgroup of index two, and thus the four representations associated to each of the two classes are affected.
\end{Rem}

\printbibliography

\end{document}